\documentclass{my_m2an} 
\usepackage{stmaryrd}
\usepackage{amsmath,amssymb}
\usepackage[mathcal]{euscript} 
\usepackage{tikz}
\usetikzlibrary{patterns} 
\usepackage{fullpage}
\usepackage{hyperref}

\usepackage{color}
\definecolor{bfonce}{rgb}{0.,0.,0.8}	
\definecolor{bclair}{rgb}{0.87,0.92,1.}
\definecolor{orangec}{rgb}{1.,0.6,0.}
\definecolor{rougef}{rgb}{0.8,0.,0.}

\newcommand{\bfa}{{\boldsymbol a}}
\newcommand{\bfe}{{\boldsymbol e}}
\newcommand{\bff}{{\boldsymbol f}}
\newcommand{\bfF}{{\boldsymbol F}}
\newcommand{\bfn}{\boldsymbol n}
\newcommand{\bfu}{{\boldsymbol u}}
\newcommand{\bfv}{{\boldsymbol v}}
\newcommand{\bfx}{\boldsymbol x}
\newcommand{\bfy}{\boldsymbol y}

\newcommand{\dx}{\ \mathrm{d}\bfx}
\newcommand{\dt}{\ \mathrm{d} t}

\newcommand{\edgespart}{{\mathfrak F}}
\newcommand{\edgespartint}{\edgespart_{\mathrm{int}}}
\newcommand{\edge}{\sigma}
\newcommand{\edges}{{\mathcal E}}
\newcommand{\edgesint}{{\mathcal E}_{{\rm int}}}
\newcommand{\mesh}{{\mathcal M}}
\newcommand{\tdisc}{{\mathcal T}}

\newcommand{\dive}{{\mathrm{div}}}
\newcommand{\diam}{{\mathrm{diam}}}
\newcommand{\gradi}{\boldsymbol \nabla}
\newcommand{\interpm}{\mathcal I\exm}

\newcommand{\mnn}{{m\in\xN}}
\newcommand{\exm}{^{(m)}}

\newtheorem{theorem}{Theorem}[section]
\newtheorem{lemma}[theorem]{Lemma}

\theoremstyle{definition}

\theoremstyle{remark}
\newtheorem{remark}[theorem]{Remark}

\begin{document}
\title[Weak consistency of finite volume schemes]
{Weak consistency of finite volume schemes \\ for systems of non linear conservation laws:  \\ extension to staggered schemes}
\author{T. Gallou\"et}
\address{I2M UMR 7373, Aix-Marseille Universit\'e, CNRS, Ecole Centrale de Marseille. 
39 rue Joliot Curie. 13453 Marseille, France. \\ (raphaele.herbin@univ-amu.fr)}

\author{R. Herbin}
\address{I2M UMR 7373, Aix-Marseille Universit\'e, CNRS, Ecole Centrale de Marseille. 
39 rue Joliot Curie. 13453 Marseille, France. \\ (raphaele.herbin@univ-amu.fr)}

\author{J.-C. Latch\'e}
\address{IRSN, BP 13115, St-Paul-lez-Durance Cedex, France (jean-claude.latche@irsn.fr)}

%

%
\begin{abstract}
We prove in this paper the weak consistency of a general finite volume convection operator acting on discrete functions which are possibly not piecewise-constant over the cells of the mesh and over the time steps.
It yields an extension of the Lax-Wendroff if-theorem for general colocated or non-colocated schemes.
This result is obtained for general polygonal or polyhedral meshes, under assumptions which, for usual practical cases, essentially boil down to a flux-consistency constraint; this latter is, up to our knowledge, novel and compares the discrete flux at a face to the mean value over the adjacent cell of the continuous flux function applied to the discrete unknown function.
We then apply this result to prove the consistency of a finite volume discretisation of a convection operator featuring a (convected) scalar variable and a (convecting) velocity field, with a staggered approximation, \ie\ with a cell-centred approximation of the scalar variable and a face-centred approximation of the velocity.
\end{abstract}
\maketitle
%
%
\section{Introduction}\label{sec:intro}

The well-known Lax-Wendroff theorem \cite{lax-60-sys} states that, on uniform 1D grids, a flux-consistent and conservative cell-centred finite-volume scheme for a system of conservation laws is weakly consistent, in the sense that the limit of any a.e. convergent sequence of $L^\infty$-bounded numerical solutions, obtained with a sequence of grids with mesh and time steps tending to zero, is a weak solution of the conservation law; it is also stated in a different form \cite[Section 12.10]{lev-02-fin}, with a BV bound assumption on the scheme.
It is generalised to non uniform 1D or Cartesian meshes in \cite[Theorem 21.2]{egh-00-fin}.
In a recent work \cite{ben-19-cons}, the Lax-Wendroff theorem is extended to obtain some error estimates for higher order schemes on uniform 1D meshes. 
The case of general (and, in particular, unstructured) discretisations has been also been tackled over the past decades: \cite{kro-96-lax}, \cite[Section 4.2.2]{god-96-num} \cite{ell-07-lax}, \cite{gal-19-wea}.
In \cite{ell-07-lax}, a quasi-uniformity assumption  is required  on the mesh, but the flux is only required to be continuous, while in \cite{gal-19-wea}, there is no uniformity assumption on the mesh but the flux is supposed to be Lipschitz continuous or at least ``lip-diag" \cite[Remark 5.2]{gal-19-wea}. 
In all the above cited works, the scheme is supposed to be colocated, in the sense that the discrete unknowns are associated to the cells of the mesh, so these results may not be used directly to cope with staggered approximations, for instance.

\medskip
The aim of this paper is to address all type of approximations, co-located or staggered; indeed, we prove the weak consistency of a general finite volume convection operator acting on general (\ie\ possibly not piecewise-constant over the cells of the mesh and over the time steps) discrete functions, under sufficient conditions which, in usual cases, turn to essentially boil down to a new flux consistency requirement; this weak consistency result is stated in Theorem \ref{th:lw} below.
The flux consistency constraint, formulated by Assertion \eqref{hyp:x}, demands a control on the difference between the discrete flux at a face (or edge) and the mean value over the adjacent cell of the continuous flux function applied to the discrete unknown function.
Theorem \ref{th:lw} is valid for general polygonal or polyhedral meshes without any supplementary assumptions on the mesh; as a by product of this work, we thus also obtain a consistency result for colocated schemes (\ie\ schemes using only piecewise-constant per cell unknowns) with possibly relaxed assumptions for the mesh compared to \cite{gal-19-wea}.
However, let us note that the proof that the assumption \eqref{hyp:x} is satisfied  is usually based on the  control of the difference between the numerical solution and its space or time translates, see \cite[Section 4]{gal-19-wea} and that these latter results may require some regularity assumptions on the mesh, see also Remark \ref{rem:flux}.

\medskip
This paper is organized as follows.
We state and prove the general consistency result in Section \ref{sec:gen_cons}.
We then apply it in Section \ref{sec:appli} to a staggered discretisation; precisely speaking, we show the consistency of a finite volume discretisation of a nonlinear convection operator for a scalar variable $\rho$ of the form $\partial_t \beta(\rho) + \dive (g(\rho)\bfu)$, where $\beta$ and $g$ are regular functions and $\bfu$ is a velocity field, and where we use a cell-centred approximation for $\rho$ and a face-centred approximation of $\bfu$.
%
%
\section{The general consistency result} \label{sec:gen_cons}

The aim is to prove the weak consistency of finite volume approximations of nonlinear convective terms which appear in most models of fluid flow. 
The general context is the following. 
Given a numerical scheme which yields some approximate solutions to the system of conservative partial differential equations, we assume that these approximate solutions converge to some functions strongly in $L^1$ , and we wish to show that the limit is indeed a solution to the system, at least in a weak sense. 
In order to do so, the usual idea is to mutiply the numerical scheme by an interpolate of a smooth function, sum over the cells of the mesh and the time steps and show that passing to the limit, we get a weak formulation of the system of partial differential equations. 
The theorem that we prove below is a mean to prove that one may indeed pass to the limit in the terms that involve nonlinear convection operators.
Let us begin with an example. 
Consider the barotropic Euler equations, which read:
\begin{subequations}\label{euler-baro}
\begin{align} & 
\partial_t \bar \rho + \dive(\bar \rho \bar \bfu) = 0,  \label{euler-baro-mass}
\\ \label{euler-baro-qdm} &
\partial_t(\bar  \rho \bar  \bfu)+ \dive(\bar \rho \bar \bfu \otimes\bar  \bfu)  + \gradi \bar p = 0,
\end{align} 
\end{subequations}
where $\bar \rho$ is the density, $\bar \bfu$ the velocity and $p$ the pressure, which, for barotropic flows, is a function of $\bar \rho$ only: $\bar p= \mathfrak p(\bar \rho)$.
Here and in the remainder of the paper, we use overlined letters when referring to the solution of the continuous problem, while non overlined letters will be used for discrete unknowns.
This system of equations is supplemented by an initial condition and suitable boundary conditions.
An entropy weak solution of the system satisfies the equations \eqref{euler-baro} and also satisfies (in a weak sense, which includes the initial condition) the following entropy condition:
\begin{equation} \label{euler-baro-entropie}
 \partial_t \bar E + \dive((\bar E  + \bar p) \bar \bfu)\le 0, \mbox{ with } \bar E = \frac 1 2 \bar \rho |\bar \bfu|^2 + \mathcal  H (\bar \rho) \mbox{ and  }  \mathcal H (s)  = s \int \dfrac{\mathfrak p (s)}{s^2} \ ds.
\end{equation} 
The weak  consistency of staggered finite volume schemes for this system of equations discretised on multi-dimensional Cartesian or unstructured meshes has been the object of several recent papers, see e.g. \cite{her-20-cons,her-18-cons}.
The system \eqref{euler-baro}-\eqref{euler-baro-entropie} may be written as
\begin{align} \label{euler-baro-C}
&\bar {\mathcal C}_1(\bar \rho,\bar \bfu)  = 0,
\\
&\bar {\mathcal C}_2(\bar \rho,\bar \bfu) + \gradi \bar p = 0,
\\
&\bar {\mathcal C}_3(\bar E,\bar \bfu)  + \dive (\bar p \bar \bfu )  \le 0,
\end{align} 
with $ \bar {\mathcal C}_1(\bar \rho,\bar \bfu)  = \partial_t \bar \rho + \dive(\bar \rho \bar \bfu) $, $\bar {\mathcal C}_2(\bar \rho,\bar \bfu)  = \partial_t(\bar  \rho  \bar \bfu)+ \dive(\bar \rho \bar \bfu \otimes\bar  \bfu) $, 
and $\bar {\mathcal C}_3(\bar E,\bar \bfu) = \partial_t \bar E + \dive(\bar E \bar \bfu)$.
In the above cited works, the system is discretised with an explicit or implicit in time scheme, and the convection operators  $\mathcal C_1$ and $\mathcal C_2$ by a first or second order finite volume scheme. 
 In fact, the system of the barotropic equations can be discretised by different schemes: explicit or implicit, colocated meshes or staggered meshes, using a Riemann solver or using an equation-by- equation procedure. 
In all cases, the consistency study will have to deal with each of the discrete non linear convection operator  $\mathcal C_i$ corresponding to $\bar {\mathcal C}_i$.
The present work aims at simplifying the proofs of consistency by giving a general result for any nonlinear convection term, discretised on colocated or staggered mesh, thereby extending our previous result of \cite{gal-19-wea} to staggered meshes. 
Theorem \ref{th:lw} below is an efficient tool to this purpose. 
We emphasize that both implicit or explicit schemes may be addressed, since the proof deals separately with the discrete time operator and the discrete space divergence operator. 

\medskip
Let us then turn to the general setting; we  suppose that:
\begin{equation}\label{pb-general}
\Omega \subset \xR^d, \; d= 1, 2,3, \; T \in (0,+\infty),
p \in \xN^\ast, \;\beta \in C^0(\xR^p, \xR),\; \bff \in C^0(\xR^p, \xR^d).
\end{equation}
We consider the conservative convection operator $\bar{\mathcal C}(\bar U)$ acting on a vector $\bar U \in \xR^p$ of functions, real-valued, and defined (in the distributional sense), for $\bar U \in L^\infty(\Omega\times(0,T), \xR^p)$, by: 
\begin{align} \label{def:op-conv}
\bar{\mathcal C} (\bar U): &\quad \Omega\times(0,T) 	\to \xR, \nonumber
\\ & \quad 
(\bfx,t) \mapsto \partial_t(\beta(\bar U))(\bfx,t) + \dive(\bff(\bar U))(\bfx,t).
\end{align}
Note that, here and throughout the paper, we use $\beta(\bar U)$ (resp. $\bff(\bar U)$ to denote the function $\beta \circ  \bar U $ obtained by composition of $\beta$ and $\bar U$ (resp. $\bff$ and $\bar U$), so, for instance, $\beta(\bar U)(\bfx,t)$ stands for $\beta(\bar U(\bfx,t))$.
In the above example of the barotropic Euler equations \eqref{euler-baro}, we have, for $i = 1,2$, $\bar{\mathcal C}_i(\bar U)  = \partial_t(\beta_i(\bar U)) +  \dive(\bff_i(\bar U))$, with  $\bar U = (\bar \rho,  \bar\bfu)$, $\beta_1( \bar U) = \bar \rho$, $\beta_2( \bar U) = \bar \rho \bar \bfu$, $\bff_1(\bar U) =  \bar \rho   \bar\bfu$ and $\bff_2(\bar U) =  \bar \rho  \bar \bfu \otimes \bar \bfu$.

\medskip
Let us denote by $\mathcal P$ a mesh of the domain, $\Omega$, consisting of a set of disjoint open polyhedral or polygonal open subsets of $\Omega$, whose union of closures is $\bar \Omega$.
To avoid cumbersome notations, we assume that any pair of adjacent cells shares a whole face, and not only a part of it; however this assumption is not necessary for the result of Theorem \ref{th:lw} to hold. 
We denote by $\delta(\mathcal P)$ the space step, defined by
\[
\delta(\mathcal P) = \max_{P \in \mathcal P} \diam (P).
\] 
Let $\edgespart$ denote the set of faces (in 3D, or edges in 2D) of the mesh, and $\edgespartint$ denote the set of faces that are not located on the boundary $\partial\Omega$; for a given polyhedron (or polygon) $P \in \mathcal P$, also called a cell, let $\edgespart(P)$ be the set of faces (or edges) of $P$.
Let $t_0 = 0 < t_1 < \ldots < t_N = T$ be a partition of $(0,T)$, denoted by $\tdisc$; 
for  such a partition  $\tdisc$, we define the time step by $\delta t = \max\,\{t_{n+1} - t_n,   n \in \llbracket 0, N-1  \rrbracket \}$, 
where  $\llbracket 0, N-1 \rrbracket$ denotes the set of integers $n$ such that $0 \le n \le  N-1.$
\medskip
The unknown is supposed to be represented by a function $U \in L^\infty(\Omega\times(0,T), \xR^p)$; we emphasize that for non colocated  schemes, some unknowns are not piecewise-constant over the cells of the mesh and over the time steps. 
For instance, when using staggered discretisations in fluid flow simulations, the velocity is discontinuous along surfaces or lines included in $P$ (see the example developed in Section \ref{sec:appli}).
The discrete convection operator that we consider here takes the following form: 
\begin{align*}
\mathcal C(U) : 
& \quad
\Omega\times(0,T) 	\to \xR,
\\ & \quad 
(\bfx,t) \mapsto
\mathcal C(U)_P^n, \mbox{ for } \bfx \in P,\ P \in \mathcal P, \mbox{ and } t \in (t_n,t_{n+1}),\ n \in \llbracket 0, N-1 \rrbracket,
\end{align*}
with
\[
\mathcal C(U)_P^n = (\eth_t \beta)_P^n + \frac 1 {|P|} \sum_{\zeta \in \edgespart(P)} |\zeta|\ \bfF_\zeta^n \cdot \bfn_{P,\zeta},
\]
where $\bigl\{\beta_P^n,\ P \in \mathcal P,\ n \in \llbracket 0, N \rrbracket \bigr\}$ is a family of real numbers,
\begin{equation}
(\eth_t \beta)_P^n = \frac{\beta_P^{n+1} - \beta_P^n}{t_{n+1} - t_n},\ n \in \llbracket 0, N-1 \rrbracket,
\label{eqdef:dtbeta}
\end{equation}
and $\bigl\{\bfF_\zeta^n,\ \zeta \in \edgespart,\ n \in \llbracket 0, N-1 \rrbracket \bigr\}$ is a family of real vectors of $\xR^d$.
Note that this form of the flux implies that the scheme is conservative.
Of course, the real numbers $\bigl\{\beta_P^n,\ P \in \mathcal P,\ n \in \llbracket 0, N \rrbracket \bigr\}$ and $\bigl\{\bfF_\zeta^n,\ \zeta \in \edgespart,\ n \in \llbracket 0, N-1 \rrbracket \bigr\}$ are related to the unknown $U$; it is the object of Theorem \ref{th:lw} below to state precisely the assumptions that must be satisfied by these quantities to ensure the consistency of the discrete convection operator.

\begin{theorem}[Weak consistency for a multi-dimensional conservative convection operator] \label{th:lw}
Under the assumptions \eqref{pb-general}, let $(\mathcal P\exm,\tdisc\exm)_\mnn$ be a sequence of possibly non uniform space-time discretisations, with $\delta(\mathcal P\exm)$ and $\delta t\exm$ tending to zero as $m\to +\infty$, and let $(U\exm)_\mnn$ be the associated sequence of discrete functions.
We suppose that the sequence $(U\exm)_\mnn$ is bounded and converges to a limit:
\begin{align} \label{lemgen:linfbound} &
\exists \ C^u \in \xR_+^\ast \mbox{ s.t. } \Vert U\exm \Vert_\infty \le C^u,\ \forall \mnn,
\\ \label{lemgen:l1conv} &
\exists \ \bar U \in L^\infty(\Omega\times(0,T), \xR^p) \mbox{ s.t. } \Vert U\exm - \bar U \Vert_{L^1(\Omega\times(0,T), \xR^p)} \to 0 \mbox{ as } m \to +\infty.
\end{align}
We also assume that the family $\{ (\beta\exm)_P^n,\ P \in \mathcal P\exm,\ n \in \llbracket 0, N\exm-1 \rrbracket,\ \mnn \}$ is bounded.
In addition, let $U_0 \in L^\infty(\Omega,\xR^p)$ and let us suppose that, as $m \to + \infty$,
\begin{align} \label{hyp:condi} &
\sum_{P \in \mathcal P_{\mathrm{int}}\exm} \int_P \bigl((\beta\exm)_P^0 - \beta(U_0)(\bfx) \bigr)\ \varphi(\bfx) \dx \to 0, \mbox{ for any } \varphi \in C_c^\infty(\Omega),
\\ \label{hyp:t} &
\sum_{n=1}^{N\exm} \sum_{P \in \mathcal P_{\mathrm{int}}\exm} \int_{t_n-1}^{t_n} \int_P \Bigl( (\beta\exm)_P^n - \beta(U\exm)(\bfx,t) \Bigr)\ \varphi(\bfx,t) \dx \dt \to 0,
\mbox{ for any } \varphi \in C_c^\infty\bigl(\Omega\times[0,T)\bigr),
\\ \label{hyp:x} &
\sum_{n=0}^{N\exm-1} \sum_{P \in \mathcal P_{\mathrm{int}}\exm} \frac {\diam(P)}{|P|}
\sum_{\zeta \in \edgespart(P)} |\zeta|\ \int_{t_n}^{t_{n+1}} \int_P \Bigl| \Big((\bfF\exm)_{\zeta}^n -\bff(U^m)(\bfx,t) \Big) \cdot \bfn_{P,\zeta} \Bigr| \dx \dt \to 0, 
\end{align}
where $\mathcal P_{\mathrm{int}}\exm$ denotes the set of cells of $ \mathcal P\exm$ that have no face or edge on the boundary $\partial \Omega$.
Then, for any $\varphi \in C_c^\infty(\Omega\times[0,T))$,
\begin{multline} \label{lw}
\int_0^T \int_\Omega \mathcal C\exm(U\exm)\ \interpm(\varphi)(\bfx,t)  \dx \dt 
\to -\int_\Omega \beta(U_0)(\bfx)\ \varphi (\bfx,0) \dx
\\
- \int_0^T \int_\Omega \Big(\beta(\bar U)(\bfx,t)\ \partial_t \varphi(\bfx,t) + \ \bff(\bar U)(\bfx,t) \cdot \gradi \varphi(\bfx,t) \Big) \dx \dt \quad \mbox{as } m \to + \infty,
\end{multline}
where $\interpm(\varphi)$ is an interpolate of $\varphi$ defined a.e. by 
\begin{multline} \label{eq:interp_phi} \qquad
\interpm(\varphi)(\bfx,t) =\varphi_P^n \mbox{ for }\bfx \in P \mbox{ and } t \in (t_n, t_{n +1}),
\\
\mbox{with }\varphi_P^n = \frac 1 {|P|} \int_P \varphi(\bfx,t_n)\dx,\quad \mbox{for } P \in \mathcal P \mbox{ and } n \in \llbracket 0, N \rrbracket.
\qquad \end{multline}
\end{theorem}
 
Before we give the proof of Theorem \ref{th:lw}, let us first briefly comment on its assumptions.

\begin{remark}[Flux consistency]\label{rem:flux}
The required flux consistency is stated by Equation \eqref{hyp:x}, which requires for the flux $(\bfF\exm)_{\zeta}^n$ through a face $\zeta$ of a cell $P$ to be close to the mean value over $P$ of the actual flux function $\bff$ applied to the unknown.
For a scheme involving only cell unknowns, for instance, the quantity $(\bfF\exm)_{\zeta}^n$ is generally a function of the unknowns in the cell $P$ and in the neighbouring cells, and checking the assumption \eqref{hyp:x} amounts to bound the difference between the unknowns and their translates.
Note that, while Theorem \ref{th:lw} holds for very general meshes, as we have already mentioned in the introduction, some regularity assumptions on the sequence of meshes may be required at this step.\\
To clarify this point, let us consider a simple one-dimensional problem for the scalar unknown $u$, with $\beta(u) = f(u) = u$, leading to the linear convection operator  $\mathcal C(u) = \partial_t u + \partial_x u$, which we discretise with the first-order upwind scheme.
Then, for $\bfx \in P$ and $t \in (t_n,t_{n+1})$, $|((\bfF\exm)_{\zeta}^n -\bff(U^m)(\bfx,t)) \cdot \bfn_{P,\zeta}| = |(u\exm)_{P^-}^n -(u\exm)_P^n|$ where $P^-$ is the left cell to $P$ when $\zeta$ is its left face, and $|((\bfF\exm)_{\zeta}^n -\bff(U^m)(\bfx,t)) \cdot \bfn_{P,\zeta}| = 0$ otherwise (disregarding the boundary cells thanks to Remark \ref{rem:bound} below).
Checking Assumption \eqref{hyp:x} thus consists in proving that the term $R\exm$ defined by
\[
R\exm = \sum_{n=0}^{N\exm-1} (t_{n+1} - t_n) \sum_{P \in \mathcal P\exm} \diam(P)\ |u^n_P -u^n_{P^-}|
\]
tends to zero  as $m$ tends to $+\infty$.
This is implied by the convergence in $L^1(\Omega \times (0,T))$ of the sequence of discrete solutions provided that the ratio $|P|/|P^-|$ is bounded independently of $m$ for the sequence of meshes under consideration \cite[Section 4]{gal-19-wea}.
A more elaborate example of application, using a staggered grid, is provided in Section \ref{sec:appli} below.
\end{remark}

\begin{remark}[Disregarding boundary cells in Assumption \eqref{hyp:x}] \label{rem:bound}
Since the support of the test function $\varphi$ is compact in $\Omega \times [0,T)$, for $\delta(\mathcal P\exm)$ small enough, $\varphi$ vanishes in the boundary cells.
Consequently, it is clear from the proof of the theorem below (see the expression \eqref{eq:def_X2} of the term $X_2\exm $) that boundary cells may be excluded in the sum in Assertion \eqref{hyp:x}.
This is the reason why only the cells in $\mathcal P_{\mathrm{int}}\exm$ are considered in Assumption \eqref{hyp:x}.
For numerical fluxes involving wider stencils (for instance in the case of higher order schemes), one could in fact reduce the set of cells involved furthermore.
\end{remark}

\begin{remark}[Regularity of $\beta$ and $\bff$] \label{rem:reg_func}
The proof of Theorem \ref{th:lw} holds if $\beta$ and $\bff$ are only continuous functions, which is the assumption made in the present section; however, to prove Assertions \eqref{hyp:t} and \eqref{hyp:x}, a locally Lipschitz continuity is often required, as in Section \ref{sec:appli}.
\end{remark}

\begin{remark}[Stronger convergence assumptions on  $\{ (\beta\exm)_{m\in \xN}\}$] \label{rem:old--theo}
In most situations, stronger convergence properties hold for $ (\beta\exm)_{m\in \xN}$, namely the weak convergence assumptions \eqref{hyp:condi} and \eqref{hyp:t} are implied by the following stronger asssumptions:
 \begin{align*}	&
 \sum_{P \in \mathcal P_{\mathrm{int}}\exm}   \int_P |(\beta\exm)_P^0  - \beta(U_0(\bfx)) | \dx  \to 0 \mbox{ as } m \to + \infty, \mbox{ with } U_0 \in L^\infty(\Omega,\xR^p),
 \\ &
 \sum_{n=0}^{N_m-1}  \sum_{P \in \mathcal P_{\mathrm{int}}\exm} \int_{t_n}^{t_{n+1}} \int_P |(\beta\exm)_P^n  - \beta(U\exm(\bfx,t)) | \dx \dt \to 0 \mbox{ as } m \to + \infty.
\end{align*}
This is the case, for instance, for the convection operator considered in Section \ref{sec:appli} below.
However, there are cases where the convergence of $\beta$ is only weak, see for instance the reconstructed kinetic energy for the full compressible Euler equations in \cite{her-20-cons}.
\end{remark}

\begin{remark}[On the interpolate of the test function] \label{rem:interpm}
Note that in the definition \eqref{eq:interp_phi} of $\interpm(\varphi)$ in \eqref{lw}, the quantities $\varphi_P^n,\ n \in \llbracket 0, N \rrbracket$, may be also defined as
\[
\varphi_P^n = \frac 1 {|P|} \int_P \varphi(\bfx,t_{n+1}) \dx,
\]
with minor changes in the arguments of the present section, essentially a slightly different assumption \eqref{hyp:t}, which reads:
\[
\sum_{n=1}^{N\exm} \sum_{P \in \mathcal P_{\mathrm{int}}\exm} \int_{t_n-1}^{t_n} \int_P \Bigl( (\beta\exm)_P^{n-1} - \beta(U\exm)(\bfx,t) \Bigr)\ \varphi(\bfx,t) \dx \dt \to 0,
\mbox{ for any } \varphi \in C_c^\infty\bigl(\Omega\times[0,T)\bigr).
\]
For instance, for a scalar problem, if the discrete function is defined as $u(\bfx,t)=u^{n-1}_P$ for $\bfx \in P$ and $t \in [t_{n-1},t_n)$ (choice often used in explicit schemes) and $\beta_P^{n-1}$ is defined in the scheme as $\beta(u^{n-1}_P)$, this assumption is trivially satisfied, since$ (\beta\exm)_P^{n-1} = \beta(U\exm)(\bfx,t)$ in $P \times (t_{n-1},t_n)$, while checking the original assumption \eqref{hyp:t} needs to bound the time translates of the discrete solution.
This is however an easy task, under a very mild regularity assumption for the time discretisation (see Section \ref{sec:appli} below).
The opposite situation occurs (\ie\ this is Assumption \eqref{hyp:t} which is now trivially satisfied) if the discrete function is defined as $u(\bfx,t)=u^n_P$ for $\bfx \in P$ and $t \in [t_{n-1},t_n)$, which is often done for implicit schemes.
\end{remark}

\begin{proof}[Proof of Theorem \ref{th:lw}]
Theorem \ref{th:lw} is the consequence of the two following lemmas, which prove respectively the convergence of the time derivative part and the space derivative part. 
Let us decompose 
\begin{align} \nonumber &
\int_0^T \int_\Omega \mathcal C\exm(U\exm)\ \interpm(\varphi)(\bfx,t)  \dx \dt = X_1\exm + X_2\exm, \mbox{ with }
\\ \label{eq:def_X1} & \hspace{20ex}
X_1\exm = \sum_{n=0}^{N\exm-1} (t_{n+1} - t_n) \sum_{P \in \mathcal P\exm } (\eth_t \beta\exm)_P^n\ \varphi_P^n,
\\ \label{eq:def_X2} & \hspace{20ex}
X_2\exm = \sum_{n=0}^{N\exm-1} (t_{n+1} - t_n) \sum_{P \in \mathcal P\exm }\ \sum_{\zeta \in \edgespart(P)} |\zeta|\ (\bfF\exm)_\zeta^n \cdot \bfn_{P,\zeta}\ \varphi_P^n.
\end{align}
Then, by Lemma \ref{lem:time-cons} below,
\[
X_1\exm \to - \int_\Omega \beta(U_0)(\bfx) \dx
- \int_0^T \int_\Omega \beta(\bar U)(\bfx,t)\ \partial_t \varphi (\bfx,t) \dx \dt \quad \mbox{as } m\to +\infty,
\]
and by Lemma \ref{lem:space-cons} below, 
\[
X_2\exm \to - \int_0^T \int_\Omega \bff(\bar U)(\bfx,t) \cdot \gradi \varphi (\bfx,t) \dx \dt 
\quad \mbox{as } m\to +\infty,
\]
which concludes the proof.
\end{proof}
%
%
\medskip
\begin{lemma}[Weak consistency, time derivative] \label{lem:time-cons}
Let the sequence $(X_1\exm)_\mnn$ be defined by \eqref{eq:def_X1}.
Then, under the assumptions and notations of Theorem \ref{th:lw}, 
\[
X_1\exm \to - \int_\Omega \beta(U_0)(\bfx)\ \varphi (\bfx,0) \dx
- \int_0^T \int_\Omega \beta(\bar U)(\bfx,t)\ \partial_t \varphi(\bfx,t) \dx \dt \quad \mbox{as } m\to +\infty.
\]
\end{lemma}

\begin{proof}
By the definition \eqref{eqdef:dtbeta} of $\eth_t^n \beta_P\exm (\bfx,t)$ and thanks to a discrete integration by parts, we get that
\[
X_1\exm = - \sum_{P \in \mathcal P\exm} |P| \ (\beta\exm)_P^0\ \varphi_P^0
- \sum_{n=1}^{N\exm} (t_n - t_{n-1}) \sum_{P \in \mathcal P\exm} |P|\ (\beta \exm)_P^n\ \frac{\varphi_P^n - \varphi_P^{n-1}}{t_n - t_{n-1}}.
\]
On the one hand, the piecewise-constant function equal to $\varphi_P^0$  on each cell $P\in \mathcal P\exm$ converges to $\varphi(\bfx,0)$ in $L^\infty(\Omega)$ as $m$ tends to $+\infty$.
On the other hand, assumption \eqref{hyp:condi} states the weak convergence, in the distributional sense, of the function $(\beta\exm)^0$ defined by $(\beta\exm)^0(\bfx)=(\beta\exm)^0$ for $\bfx \in P,\ P\in \mathcal P\exm$ to the function $\beta(U_0)$.
In addition, $(\beta\exm)^0$ is supposed to be bounded.
We thus have:
\[
- \sum_{P \in \mathcal P\exm} |P| \ (\beta\exm)_P^0\ \varphi_P^0 \to - \int_\Omega \beta(U_0)(\bfx)\ \varphi (\bfx,0) \dx \quad \mbox{as } m\to +\infty.
\]
Let the piecewise constant function $\eth_t\exm \varphi : \Omega\times(0,T) \to \xR^d$ be defined by 
\[\eth_t\exm \varphi(\bfx,t) = \dfrac{\varphi_P^{n+1} - \varphi_P^n}{t_{n+1} - t_n} \mbox{ for }(\bfx,t) \in P \times (t_n,t_{n+1}).\]
The function $\eth_t\exm \varphi$ converges uniformly to $\partial_t \varphi$ in $L^\infty(\Omega \times (0,T))$.
The second term of $X_1\exm$ may be decomposed as
\[
-\sum_{n=1}^{N\exm} (t_n - t_{n-1}) \sum_{P \in \mathcal P\exm} |P|\ (\beta \exm)_P^n\ \frac{\varphi_P^n - \varphi_P^{n-1}}{t_n - t_{n-1}}= Y_1\exm + Y_2\exm  
\]
with
\begin{align*} 
& 
Y_1\exm = -\sum_{n=1}^{N\exm} \sum_{P \in \mathcal P\exm} \int_{t_{n-1}}^{t_n} \int_P \Bigl( (\beta\exm)_P^n - \beta(U\exm)(\bfx,t) \Bigr)\ \eth_t\exm \varphi(\bfx,t) \dx \dt,
\\ &
Y_2\exm = - \int_0^T \int_\Omega \beta(U\exm)(\bfx,t)\ \eth_t\exm \varphi(\bfx,t) \dx \dt.
\end{align*}
Invoking the assumption \eqref{hyp:t} and the uniform convergence of $\eth_t\exm \varphi$ to $\partial_t\varphi$ , we thus get that the sequence $(Y_1\exm)_\mnn$ tends to zero.
On the other hand, thanks to the assumptions \eqref{lemgen:linfbound}, \eqref{lemgen:l1conv} and the regularity of $\beta$, we get that
\[
\lim_{m \to +\infty} X_1\exm = \lim_{m \to +\infty} Y_2\exm = 
- \int_0^T \int_\Omega \beta(\bar U)(\bfx,t)\ \partial_t \varphi(\bfx,t) \dx \dt.
\]
\end{proof}
%
%

\medskip
\begin{lemma}[Weak consistency, space derivative] \label{lem:space-cons}
Let the sequence $(X_2\exm)_\mnn$ be defined by \eqref{eq:def_X2}.
Then, under the assumptions and notations of Theorem \ref{th:lw},
\[
X_2\exm \to - \int_0^T \int_\Omega \bff(\bar U)(\bfx,t) \cdot \gradi \varphi(\bfx,t) \dx \dt \quad \mbox{as } m\to +\infty.
\]
\end{lemma}

\begin{proof}
Since $\varphi$ is compactly supported and since $\delta (\mathcal P\exm) \to 0$ as $m \to 0$, there exists $M \in \xN$ such that for $m \ge M$, $\varphi_P^n = 0$ for all $\bfx \in \mathcal P\exm \setminus \mathcal P_{\mathrm{int}}\exm$.
Moreover, since for a face $\zeta$ separating $P$ and $P'$, one has $\bfn_{P,\zeta}= - \bfn_{P',\zeta}$, we get that
\[
X_2\exm
= \sum_{n=0}^{N\exm-1} (t_n - t_{n-1}) \sum_{P \in \mathcal P_{\mathrm{int}}\exm}\ \sum_{\zeta \in \edgespart(P)} |\zeta| \ (\bfF\exm)_\zeta^n \cdot \bfn_{P,\zeta}\ \varphi_P^n
= \sum_{n=0}^{N\exm-1} (t_n - t_{n-1}) \sum_{P_{\mathrm{int}} \in \mathcal P\exm} A_P^n
\]
with
\[
A_P^n =\sum_{\zeta \in \edgespart(P)} |\zeta|\ (\bfF\exm)_{\zeta}^n \cdot \bfn_{P,\zeta} \ (\varphi_P^n -\varphi_\zeta^n),
\]
where $\varphi_\zeta^n$ denotes the mean value of $\varphi(\bfx,t_n)$ over $\zeta$.
Now, for any $\bfx \in P$, $t \in [t_n, t_{n+1})$, we can decompose $A_P^n$ as $A_P^n = B_P^n(\bfx,t)+R_P^n(\bfx,t)$ with
\begin{align*} &
B_P^n(\bfx,t) = \sum_{\zeta \in \edgespart(P)} |\zeta|\ \bff(U\exm)(\bfx,t) \cdot \bfn_{P,\zeta} \ (\varphi_P^n -\varphi_\zeta^n),
\\ &
R_P^n(\bfx,t) = \sum_{\zeta \in \edgespart(P)} |\zeta|\ \Bigl( (\bfF\exm)_\zeta^n - \bff(U\exm)(\bfx,t) \Bigr) \cdot \bfn_{P,\zeta} \ (\varphi_P^n -\varphi_\zeta^n).
\end{align*}
Since $\displaystyle \sum_{\zeta \in \edgespart(P)} |\zeta|\ \bfn_{P,\zeta} =0$, we have
\begin{multline} \label{bnp}
B_P^n(\bfx,t) = -\sum_{\zeta \in \edgespart(P)} |\zeta|\ \bff(U\exm)(\bfx,t) \cdot \bfn_{P,\zeta} \ \varphi_\zeta^n
= - |P|\ \bff(U\exm)(\bfx,t) \cdot (\gradi \varphi)_P^n, 
\\
\mbox{with } (\gradi \varphi)_P^n = \dfrac 1 {|P|} \sum_{\zeta \in \edgespart(P)} |\zeta|\ \varphi_\zeta^n\ \bfn_{P,\zeta} = \dfrac 1 {|P|} \int_P \gradi \varphi(\bfx,t_n) \dx.
\end{multline}
Note that the piecewise constant function $\gradi\exm \varphi:\ \Omega\times(0,T) \to \xR^d$ defined by
\[
  \gradi\exm \varphi(\bfx,t) = (\gradi \varphi)_P^n \mbox{ for } (\bfx,t) \in P \times (t_n,t_{n+1})
\]
converges uniformly to $\gradi \varphi$ in $L^\infty(\Omega \times (0,T))^d$. 
Since, by definition of $B_P^n(\bfx,t)$ and $R_P^n(\bfx,t)$,
\[
A_P^n = \frac 1 {(t_{n+1} - t_n) \ |P|} \Bigl(\int_{t_n}^{t_{n+1}} \int_P B_P^n(\bfx,t) \dx \dt + \int_{t_n}^{t_{n+1}} \int_P R_P^n(\bfx,t) \dx \dt \Bigr),
\]
we get
\begin{multline} \label{eq:X2fBP}
X_2\exm =\sum_{n=0}^{N\exm-1} \sum_{P \in \mathcal P_{\mathrm{int}}\exm} \frac 1 {|P|} 
\Bigl(\int_{t_n}^{t_{n+1}} \int_P B_P^n(\bfx,t) \dx \dt + \int_{t_n}^{t_{n+1}} \int_P R_P^n(\bfx,t) \dx \dt \Bigr)
\\
= -\int_0^T \int_\Omega \bff(U\exm)(\bfx,t) \cdot \gradi\exm \varphi(\bfx,t) \dx \dt + 
\sum_{n=0}^{N\exm-1} \sum_{P \in \mathcal P_{\mathrm{int}}\exm} \frac 1 {|P|} \int_{t_n}^{t_{n+1}} \int_P R_P^n(\bfx,t) \dx \dt.
\end{multline}
Owing to the boundedness and convergence assumptions on $U\exm$ and to the uniform convergence of $\gradi\exm \varphi$ to $\gradi \varphi$, the first term tends to $\displaystyle -\int_0^T \int_\Omega \bff(\bar U)(\bfx,t) \cdot \gradi \varphi(\bfx,t) \dx \dt$ as $m \to +\infty$.
Since $|\varphi_\zeta^n - \varphi_P^n | \leq C_\varphi\, \diam(P)$, with $C_\varphi$ depending only on $\varphi$, we get, for any $\bfx \in P$ and $t \in (t_n,t_{n+1})$,
\[
|R_P^n(\bfx,t)| \leq C_\varphi\ \sum_{\zeta \in \edgespart(P)} |\zeta|\ \Bigl| \Bigl( (\bfF\exm)_\zeta^n - \bff(U\exm)(\bfx,t) \Bigr) \cdot \bfn_{P,\zeta} \Bigr|\ \diam(P).
\]
The second term of the right-hand side of Relation \eqref{eq:X2fBP} thus tends to $0$ as $m \to +\infty$ thanks to the assumption \eqref{hyp:x}, which concludes the proof.
\end{proof}
%
%
\section{An example of application for staggered discretisations} \label{sec:appli}

The interest of Theorem \ref{th:lw} lies in the fact that it may deal with terms combining several variables, associated to different meshes and time discretisations.
A typical exemple of a such a case is the balance equation for the entropy in barotropic compressible flows \eqref{euler-baro-entropie}, where the entropy $E$ is a nonlinear function of the density $\rho$ and the velocity $\bfu$ which, in staggered discretisation, are approximated on different meshes, and may also be evaluated at different time levels.
Hence, Theorem \ref{th:lw} is a suitable tool to prove the consistency of this equation.
In this section, we focuss on a similar but simpler problem, namely a staggered discretisation of a convection operator combining the time derivative of the function of a single scalar variable and a space divergence term, with a flux obtained as the product of another function of this scalar variable with the velocity.

\medskip
We suppose that $\Omega$ is an open bounded polyhedral set of $\xR^2$, and consider the following convection operator:
\begin{align} \label{def:op-conv_ex}
\mathcal C(\bar U) : &\quad \Omega\times(0,T) 	\to \xR,  \nonumber
\\ & \quad 
(\bfx,t) \mapsto \partial_t(\beta(\bar q))(\bfx,t) + \dive\bigl(g(\bar q)\,\bar \bfv \bigr)(\bfx,t),
\end{align}
with $\bar U = (\bar q,\bar \bfv) \ : \Omega \times (0,T) \rightarrow \xR \times \xR^2$, $\bff(\bar U) = \bff(\bar q,\bar \bfv) = g(\bar q)\,\bar \bfv$, where $\beta:\ \xR \rightarrow \xR$ and $g :\ \xR \rightarrow \xR$ are locally Lipschitz-continuous real functions.
Note that, for instance, the convection term of Equation \eqref{euler-baro-mass} may be written as \eqref{def:op-conv_ex} with $\bar U = (\bar \rho, \bar \bfu)$, $\beta(s) = s$ and $g(s) = s$.

\medskip
In order to discretise this convection operator, we consider two types of staggered arrangements. 
In both arrangements, the scalar unknowns are located at the center of the cells. 
However, they differ in the use of the vector unknowns.
The first discretisation uses the whole velocity vector unknown on each edge of the mesh; this corresponds to the Rannacher-Turek (RT) discrete unknowns in the finite element setting \cite{ran-92-sim}.
The second discretisation uses only the normal component of the velocity on each edge; this latter arrangement of the discrete unknowns is very often referred to as the Marker-and-Cell (MAC) scheme \cite{har-65-num}.
Hence we will refer to the first arrangement as the RT case, and the second as the MAC case.
Such discretisations are called staggered and are widely used in computational fluid dynamics; an example of the implementation of a staggered discretisation for the solution of the barotropic and full Euler equations may be found {\it e.g.} in \cite{her-18-cons, her-20-cons}.

\medskip
We suppose that the mesh is composed either of general quadrangles (RT case), or of rectangles (MAC case).
We recall that $\edgespart$ stands for the set of edges of the mesh, and the internal edge separating the cells $P$ and $Q$ is denoted by $\zeta=P|Q$.
This mesh will be referred to in the following as the primal mesh.

\medskip
We also introduce now one or two dual meshes, depending on the case.
\begin{list}{-}{\itemsep=0.5ex \topsep=0.5ex \leftmargin=1.cm \labelwidth=0.3cm \labelsep=0.5cm \itemindent=0.cm}
\item{\it  RT case - }
In this case, the (unique) dual mesh consists in  a new partition of $\Omega$ indexed by the elements of $\edgespart$, \ie\ $\Omega=\cup_{\zeta \in \edgespart} D_\zeta$.
For an internal edge $\zeta=P|Q$, the set $D_\zeta$ is supposed to be a subset of $P \cup Q$ and we define $D_{P,\zeta}=D_\zeta \cap P$, so that $D_\zeta=D_{P,\zeta} \cup D_{Q,\zeta}$ (see Figure \ref{fig:rt-space_disc}); for an external edge $\zeta$ of a cell $P$, $D_\zeta$ is a subset of $P$, and $D_\zeta=D_{P,\zeta}$.
The cells $(D_\zeta)_{\zeta \in \edgespart}$ are referred to as the dual or diamond cells, and $D_{P,\zeta}$ as half dual cells or half diamond cells.
For a rectangular cell $P$, we define $D_{P,\zeta}$ as the simplex having the mass center of $P$ as vertex and the edge $\zeta$ as basis; this definition is extended to general primal meshes by supposing that $|D_{P,\zeta}|$ is still equal to $|P|/4$ and that the sub-cells connectivities (\ie\ the way the half-dual cells share a common edge) is left unchanged.
Note that the actual geometry of the dual cells does not need to be specified (and a dual cell may not be a polytope, a dual edge being possibly curved).

\begin{figure}[htb]
\begin{center}
\scalebox{0.9}{
\begin{tikzpicture} 
\fill[color=bclair,opacity=0.3] (1.,2.) -- (4.,1.) -- (5.,3.) -- (3.,4.) -- (1.,2.);
\fill[color=orangec!30!white,opacity=0.3] (5.,3.) -- (6.,5.) -- (3.,6.) -- (3.,4.);
\draw[very thick, color=black] (1.,2.) -- (4.,1.) -- (5.,3.) -- (3.,4.) -- (1.,2.);
\draw[very thick, color=black] (5.,3.) -- (6.,5.) -- (3.,6.) -- (3.,4.);
\draw[very thick, color=black] (1.,2.) -- (0.,1.5);
\draw[very thick, color=black] (4.,1.) -- (4.5,0.);
\draw[very thick, color=black] (3.,4.) -- (1.5,4.2);
\draw[very thick, color=black] (5.,3.) -- (6,3);
\draw[very thick, color=black] (6.,5.) -- (6.,6.);
\draw[very thick, color=black] (6.,5.) -- (7.,5.);
\draw[very thick, color=black] (3.,6.) -- (3.5,7.);
\draw[very thick, color=black] (3.,6.) -- (2.,6.);
\node at (3., 2.5){$\mathbf P$}; \node at (4.5, 4.7){$\mathbf Q$};
\node at (3., 2.5){$\mathbf P$}; \node at (4.5, 4.7){$\mathbf Q$};
\draw[very thick, color=black] (5.,3.) -- (3.,4.) node[midway,sloped,above]{$\zeta=P|Q$};
\fill[color=bclair,opacity=0.3] (10.,4.) .. controls (10.1,3.) .. (10.3,2.2) .. controls (11.,2.6) .. (12.,3.);
\fill[color=orangec!30!white,opacity=0.3] (10.,4.) .. controls (10.8,5.) .. (11.2,4.8) .. controls (11.6,4.7) .. (12.,3.);
\draw[very thick, color=black] (8.,2.) -- (11.,1.) -- (12.,3.) -- (10.,4.) -- (8.,2.);
\draw[very thick, color=black] (12.,3.) -- (13.,5.) -- (10.,6.) -- (10.,4.);
\draw[very thick, color=black] (8.,2.) -- (7.,1.5);
\draw[very thick, color=black] (11.,1.) -- (11.5,0.);
\draw[very thick, color=black] (10.,4.) -- (8.5,4.2);
\draw[very thick, color=black] (12.,3.) -- (13,3);
\draw[very thick, color=black] (13.,5.) -- (14.,5.);
\draw[very thick, color=black] (13.,5.) -- (13.,6.);
\draw[very thick, color=black] (10.,6.) -- (10.5,7.);
\draw[very thick, color=black] (10.,6.) -- (9.,6.);
\draw[thick, color=bfonce] (10.,4.) .. controls (10.1,3.) .. (10.3,2.2);
\draw[thick, color=bfonce] (10.3,2.2) .. controls (11.,2.6) .. (12.,3.);
\draw[thick, color=bfonce] (8.,2.) .. controls (9.,2.3) .. (10.3,2.2);
\draw[thick, color=bfonce] (11.,1.) .. controls (10.8,2.) .. (10.3,2.2);
\draw[thick, color=orangec!80!black] (10.,4.) .. controls (10.8,5.) .. (11.2,4.8) .. controls (11.6,4.7) .. (12.,3.);
\draw[thick, color=orangec!80!black] (13.,5.) -- (11.2,4.8);
\draw[thick, color=orangec!80!black] (10.,6.) -- (11.2,4.8);
\draw[very thick, color=black] (12.,3.) -- (10.,4.) node[midway,sloped,above]{$D_{Q,\zeta}$} node[midway,sloped,below]{$D_{P,\zeta}$};
\end{tikzpicture}
}
\end{center}
\caption{Primal and dual meshes and associated notations for the quadrangular mesh and Rannacher-Turek like unknowns.}
\label{fig:rt-space_disc}
\end{figure}
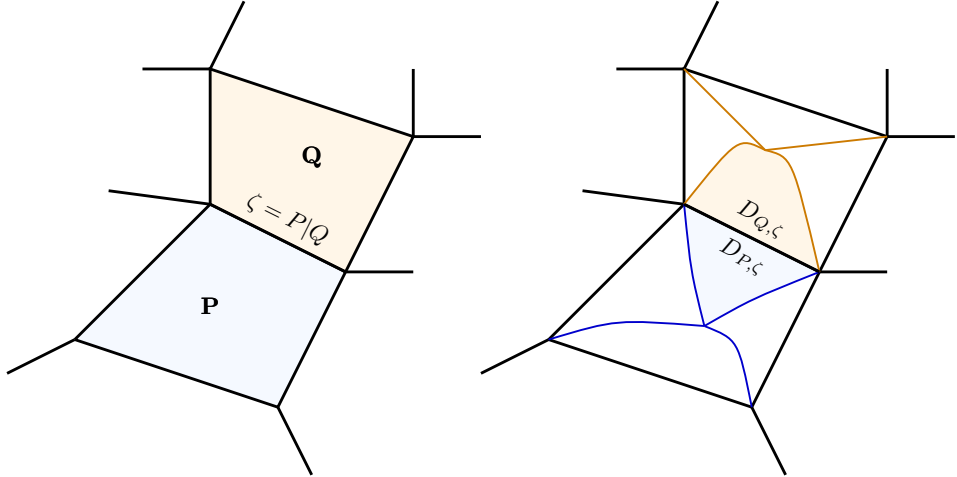

\item{\it MAC case -}
In this case, two dual meshes are considered, each  consisting in a partition of $\Omega$ indexed by the vertical and horizontal elements of $\edgespart$, \ie\ $\Omega=\cup_{\zeta \in \edgespart^{(i)}} D_\zeta$, $i=1,2$, where $\edgespart^{(1)}$ (resp. $\edgespart^{(2)}$) denotes the set of vertical (resp. horizontal) edges.
The cells $(D_\zeta)_{\zeta \in \edgespart}$ are still referred to as the dual cells.
They are no longer diamond shaped; indeed, a half dual cell $D_{P,\zeta}$ is now half of the rectangle $P$ with side $\zeta$ (see Figure \ref{fig:mac-space_disc}). 
As in the former case, for an internal edge $\zeta=P|Q$, the dual cell $D_\zeta$ is the subset of $P \cup Q$ defined as $D_\zeta=D_{P,\zeta} \cup D_{Q,\zeta}$; for an external edge $\zeta$ of a cell $P$, $D_\zeta$ is a subset of $P$, and $D_\zeta=D_{P,\zeta}$.

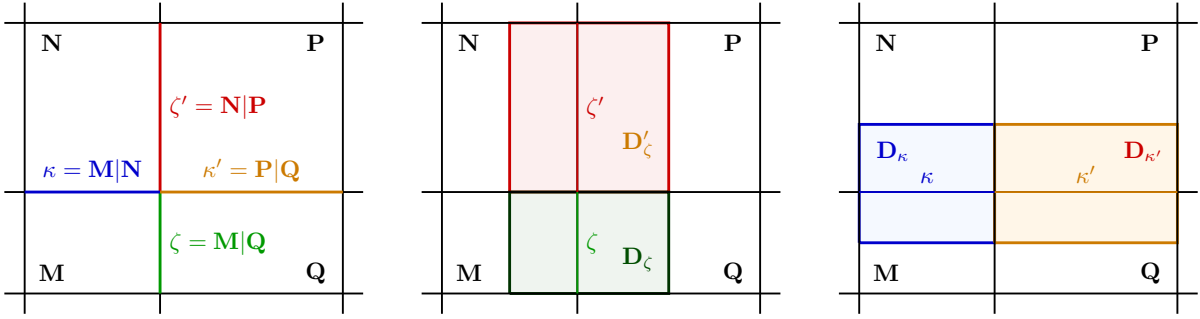
\begin{figure}[htb]
\begin{center}
\scalebox{0.9}{
%
\begin{tikzpicture} 
\draw[thick, color=black] (0.7, 0.5) -- (6.,0.5) ;
\draw[thick, color=black] (0.7, 2.) -- (6.,2.) ;
\draw[thick, color=black] (0.7, 4.5) -- (6.,4.5) ;
\draw[thick, color=black] (1.,0.2) -- (1.,4.8) ;
\draw[thick, color=black] (3.,0.2) -- (3., 4.8) ;
\draw[thick, color=black] (5.7,0.2) -- (5.7,4.8) ;
\node at (1.4,0.8){$\mathbf M$}; \node at (1.4,4.2){$\mathbf N$}; \node at (5.3,4.2){$\mathbf P$}; \node at (5.3,0.8){$\mathbf Q$};
%
\draw[very thick, color=green!60!black] (3.,0.5) -- (3.,2.) node[midway, right]{$ \mathbf \zeta = \mathbf M | \mathbf Q$} ;
\draw[very thick, color=rougef] (3.,2.) -- (3.,4.5) node[midway, right]{$ \mathbf \zeta' = \mathbf N | \mathbf P$} ;
\draw[very thick, color=bfonce] (1.,2.) -- (3.,2.) node[midway, above]{$ \mathbf \kappa = \mathbf M | \mathbf N$} ;
\draw[very thick, color=orangec!80!black] (3.,2.) -- (5.7,2.) node[midway, above]{$ \mathbf \kappa' = \mathbf P| \mathbf Q$} ;
\end{tikzpicture}
%
\hspace{0.6cm} 
\begin{tikzpicture} 
\fill[color=rougef!20!white,opacity=0.3] (2.,2.) -- (2.,4.5) -- (4.35,4.5) -- (4.35,2.) -- (2.,2.);
\draw[very thick, color=rougef] (2.,2.) -- (2.,4.5) -- (4.35,4.5) -- (4.35,2.) -- (2.,2.);
\fill[color=green!30!black!20!white,opacity=0.3] (2.,0.5) -- (2.,2.) -- (4.35,2.) -- (4.35,0.5) -- (2.,0.5);
\draw[very thick,   color=green!20!black] (2.,0.5) -- (2.,2.) -- (4.35,2.) -- (4.35,0.5) -- (2.,0.5);
\node[color=green!30!black] at (3.9, 1.){$ {\mathbf D_\zeta }$};
\node[color=orangec!80!black] at (3.9, 2.7){$ {\mathbf D_\zeta' }$};

\draw[thick, color=black] (0.7, 0.5) -- (6.,0.5) ;
\draw[thick, color=black] (0.7, 2.) -- (6.,2.) ;
\draw[thick, color=black] (0.7, 4.5) -- (6.,4.5) ;
\draw[thick, color=black] (1.,0.2) -- (1.,4.8) ;
\draw[thick, color=black] (3.,0.2) -- (3., 4.8) ;
\draw[thick, color=black] (5.7,0.2) -- (5.7,4.8) ;
\node at (1.4,0.8){$\mathbf M$}; \node at (1.4,4.2){$\mathbf N$}; \node at (5.3,4.2){$\mathbf P$}; \node at (5.3,0.8){$\mathbf Q$};
%
\draw[thick, color=rougef] (3.,2.) -- (3.,4.5) node[midway,right]{$\zeta'$} ;
\draw[thick, color=green!60!black] (3.,0.5) -- (3.,2.) node[midway,right]{$\zeta$} ;
\end{tikzpicture}
%
\hspace{0.6cm} 
\begin{tikzpicture} 
\fill[color=bclair,opacity=0.3] (1.,1.25) -- (1.,3.) -- (3.,3.) -- (3.,1.25) -- (1.,1.25);
\fill[color=orangec!30!white,opacity=0.3] (3.,1.25) -- (3.,3.) -- (5.7,3.) -- (5.7,1.25) -- (3.,1.25);
\draw[very thick,  color=bfonce]  (1.,1.25) -- (1.,3.) -- (3.,3.) -- (3.,1.25) -- (1.,1.25);
\draw[very thick,  color=orangec!80!black] (3.,1.25) -- (3.,3.) -- (5.7,3.) -- (5.7,1.25) -- (3.,1.25);
\node[color=bfonce] at (1.5, 2.6){$ {\mathbf D_\kappa }$};
\node[color=rougef] at (5.2, 2.6){$ {\mathbf D_{\kappa' }}$};

\draw[thick, color=black] (0.7, 0.5) -- (6.,0.5) ;
\draw[thick, color=black] (0.7, 2.) -- (6.,2.) ;
\draw[thick, color=black] (0.7, 4.5) -- (6.,4.5) ;
\draw[thick, color=black] (1.,0.2) -- (1.,4.8) ;
\draw[thick, color=black] (3.,0.2) -- (3., 4.8) ;
\draw[thick, color=black] (5.7,0.2) -- (5.7,4.8) ;
\node at (1.4,0.8){$\mathbf M$}; \node at (1.4,4.2){$\mathbf N$}; \node at (5.3,4.2){$\mathbf P$}; \node at (5.3,0.8){$\mathbf Q$};
%
\draw[thick, color=bfonce] (1.,2.) -- (3.,2.) node[midway,above]{$\kappa $} ;
\draw[thick, color=orangec!80!black] (3.,2.) -- (5.7,2.) node[midway,above]{$\kappa'$} ;
\end{tikzpicture}
}
\end{center}
\caption{Primal and dual meshes and associated notations for the MAC case. - Left: the primal cells; the edges $\zeta$ and $\zeta'$ belong to $\edgespart^{(1)}$ and the edges $\kappa$ and $\kappa'$  to $\edgespart^{(2)}$. - Center: the dual cells associated to  $\edgespart^{(1)}$. - Right: the dual cells associated to  $\edgespart^{(2)}$.
}
\label{fig:mac-space_disc}
\end{figure}

\end{list}

\medskip
The  scalar unknown $q$ is associated to the primal cells:
\begin{equation*}  
q(\bfx,t) = q_P^n \quad \mbox{for } \bfx \in P,\ P \in \mathcal P, \ t \in [t_n,t_{n+1}),\ n \in \llbracket 0, N-1 \rrbracket.
\end{equation*}
The unknowns corresponding to the vector-valued unknown $\bfv$ are located at the center of the edges in the RT case; in the MAC case, the unknowns associated to the $i$-th component of $\bfv$ are located at the center of the edges of the $i$-th dual mesh.
Hence the associated approximate vector function reads:
\begin{list}{-}{\itemsep=0.5ex \topsep=0.5ex \leftmargin=1.cm \labelwidth=0.3cm \labelsep=0.5cm \itemindent=0.cm}
\item RT case -- the whole vector unknown is associated to each dual cell :
\[
\bfv(\bfx,t) = \bfv_\zeta^n \quad \mbox{for } \bfx \in D_\zeta,\ \zeta \in \edgespart,\ t \in [t_n,t_{n+1}),\ n \in \llbracket 0, N-1 \rrbracket.
\]
\item MAC case -- The $i$-th component of the vector unknown is associated to the cells of the $i$-th dual mesh, so that $\bfv(\bfx,t) = (v_1(\bfx,t),\ v_2(\bfx,t))^t$  where, for $i=1,\ 2$,
\[  
v_i(\bfx,t) = v_\zeta^n, \mbox{ for } \bfx \in D_\zeta,\ \zeta \in \edgespart^{(i)} \mbox{ and } t \in [t_n,t_{n+1}),\ n \in \llbracket 0, N-1 \rrbracket.
\]
\end{list}

\medskip
Let $\bfe^{(i)}$ denote the $i$-th unit vector; with the notations of the previous section, the considered discrete convection operator reads:
\begin{multline*}
\mathcal C_{\mathcal P}(q,\bfv)_P^n = (\eth_t \beta)_P^n + \frac 1 {|P|} \sum_{\zeta \in \edgespart(P)} |\zeta|\ \bfF_\zeta^n \cdot \bfn_{P,\zeta},\mbox{ with } 
\beta_P^n = \beta(q_P^n)\mbox{ and } \bfF_\zeta^n =  \bff(q_\zeta^n,\bfv_\zeta^n ) =   g(q_\zeta^n) \ \bfv_\zeta^n  
\\
\mbox{ where }\bfv_\zeta^n  \mbox { is } 
\begin{cases}
\mbox {  the vector of discrete unknowns} & \mbox { in the RT case}, \\
\mbox {  defined as } v_\zeta^n  \ \bfe^{(i)} \mbox{ for } \zeta \in \edgespart^{(i)},\ i = 1 \mbox{ or } 2, & \mbox{ in the MAC case},
\end{cases}
\end{multline*}
and, for $\zeta=P|Q$, $q_\zeta^n$ stands for a convex combination of $q_P^n$ and $q_Q^n$.
The initial value for the scalar unknowm $q$ is defined by
\begin{equation}\label{eq:cond_ini}
q_P^0 = \frac 1 {|P|}\ \int_P q_0(\bfx) \dx. 
\end{equation}

\medskip
The consistency result for the discrete convection operator is given in the next lemma; it uses the following regularity parameters of the mesh:
\[
\theta_1(\mathcal P) = \max_{P \in \mathcal P} \frac{\diam(P)^2}{|P|}, \quad
\theta_2(\mathcal P) = \max \Bigl\{ \frac{|P|}{|Q|},\ P \mbox{ and } Q \mbox{ adjacent cells of } \in \mathcal P \Bigr\}.
\]
Note that in the MAC case (in fact, for a Cartesian grid), the regularity parameter $\theta_1(\mathcal P)$ controls the ratio between the two dimensions (\ie\ the height and the width) of a cell.
For a rectangular computational domain, with thus observe that the ratio $|\zeta|/|\zeta'|$, for $(\zeta,\zeta') \in (\edgespart^{(i)})^2$, $i = 1$, $2$, is bounded by $\theta_1(\mathcal P)^2$, which is a quasi-uniformity property of the mesh.
This also implies that $\theta_2(\mathcal P) \leq \theta_1(\mathcal P)^2$, and so the second regularity parameter is useless.
It may easily be checked that similar relations holds for a general MAC scheme, \ie\ a union of matching Cartesian grids, with powers of $\theta_1(\mathcal P)$ possibly higher than $2$.
Hence, the regularity of a MAC mesh (or of a Cartesian grid) may be equivalently measured by
\[
\theta(\mathcal P) = \max \Bigl \{\frac{\bar h^{(1)}}{\underline h^{(2)}},\ \frac{\bar h^{(2)}}{\underline h^{(1)}} \Bigr \},
\]
with, for $i = 1$, $2$, $\bar h^{(i)}=\max \{ |\zeta|,\ \zeta \in \edgespart^{(i)} \}$ and $\underline h^{(i)}=\min \{ |\zeta|,\ \zeta \in \edgespart^{(i)} \}$.

\medskip
We also measure the regularity of the time discretisation by the parameter $\theta_3(\tdisc)$ defined by
\[
\theta_3(\tdisc) = \max_{1 \leq n \leq N-1} \Bigl\{ \frac{t_{n+1}-t_n}{t_n-t_{n-1}}, \frac{t_n-t_{n-1}}{t_{n+1}-t_n} \Bigl\}.
\]

\begin{lemma}[Consistency] \label{lem:cons_stag}
Let a sequence of discretisations $(\mathcal P\exm)_\mnn$ and $(\tdisc\exm)_\mnn$ be given, with $\delta(\mathcal P\exm)$ and $\delta t\exm$ tending to zero, and let $(q\exm,\bfv\exm)_\mnn$ be the associated sequence of discrete functions.
We suppose that
\begin{equation}\label{eq:hyp_reg}
\exists\  \theta \in \xR \mbox{ such that } \max \{\theta_1(\mathcal P\exm),\ \theta_2(\mathcal P\exm),\ \theta_3(\tdisc\exm),\ m \in \xN\} \leq   \theta.
\end{equation}
We suppose that the sequences $(q\exm)_\mnn$ and $(\bfv\exm)_\mnn$ are bounded in $L^\infty(\Omega \times (0,T))$ and $L^\infty(\Omega \times (0,T))^2$ respectively, and that, when $m$ tends to $+\infty$,  they converge  in $L^p(\Omega \times (0,T))$ and $L^p(\Omega \times (0,T))^2$, $1 \leq p < +\infty$, to $\bar q \in L^\infty(\Omega \times (0,T))$ and $\bar \bfv \in L^\infty(\Omega \times (0,T))^2$ respectively.
Then, for any function $\varphi \in C^\infty_c(\Omega \times [0,T))$,
\begin{multline} \label{lw_ex}
\int_0^T \int_\Omega \mathcal C\exm(U\exm)(\bfx,t)\ \interpm(\varphi) \dx \dt \to
-\int_\Omega \beta(q_0)(\bfx)\ \varphi (\bfx,0) \dx
\\
- \int_0^T \int_\Omega \Bigl(\beta(\bar q)(\bfx,t)\ \partial_t \varphi(\bfx,t) + \ \bigl(g(\bar q)\, \bar \bfv)\bigr)(\bfx,t) \cdot \gradi \varphi(\bfx,t) \Bigr) \dx \dt
\quad \mbox{as } m \to + \infty.
\end{multline}
\end{lemma}

\begin{proof}
In this proof, we denote by $C_\beta$ and $C_g$ the Lipschitz modulus of $\beta$ and $g$ respectively on the interval $[\underline q, \bar q]$, where $\underline q \in \xR$ and $\bar q \in \xR$ are such that 
\[
\underline q \leq (q\exm)_P^n\leq \bar q,\ \forall P \in \mathcal P\exm,\ n \in \llbracket 0, N\exm \rrbracket,\ \forall m \in \xN.
\]

The proof of this lemma relies on Theorem \ref{th:lw}.
The consistency of the initialization with the initial condition (Assumption \eqref{hyp:condi}) follows from its definition \eqref{eq:cond_ini}; indeed, for any $\varphi \in C_c^\infty(\Omega)$,
\[
\Bigl| \sum_{P \in \mathcal P\exm} \int_P \bigl((\beta\exm)_P^0 - \beta(q_0)(\bfx) \bigr)\ \varphi(\bfx) \dx \Bigr|
\leq C_\beta\ \Vert \varphi \Vert_{L^\infty(\Omega)} \sum_{P \in \mathcal P\exm} \int_P |q_0(\bfx)-q^0_P|,
\]
and thus tends to zero for any function $q_0 \in L^1(\Omega)$.
Since $(\beta\exm)_P^n = \beta((q\exm)_P^n)$, the left-hand side of Assertion \eqref{hyp:t} reads, with $\varphi \in C_c^\infty(\Omega\times[0,T))$:
\begin{multline*}
R\exm_t=\sum_{n=1}^{N\exm} \sum_{P \in \mathcal P_{\mathrm{int}}\exm} \int_{t_n-1}^{t_n} \int_P \Bigl( (\beta\exm)_P^n - \beta(U\exm)(\bfx,t) \Bigr)\ \varphi(\bfx,t) \dx \dt
\\
= \sum_{n=1}^{N\exm} \sum_{P \in \mathcal P_{\mathrm{int}}\exm} \int_{t_n-1}^{t_n} \int_P \Bigl( \beta\bigl((q\exm)_P^n\bigr) - \beta\bigl( (q\exm)_P^{n-1} \bigr) \Bigr)\ \varphi(\bfx,t) \dx \dt.
\end{multline*}
We thus have
\[
|R\exm_t| \leq C_\beta \ \Vert \varphi \Vert_{L^\infty(\Omega\times[0,T))} \sum_{n=1}^{N\exm} (t_n - t_{n-1}) \sum_{P \in \mathcal P_{\mathrm{int}}\exm} |(q\exm)_P^n - (q\exm)_P^{n-1}|,
\]
and thus $R\exm_t$ tends to zero thanks to the assumed regularity of the sequence of time discretisations, invoking the bound of the time-translates of a converging sequence of functions of $L^1(\Omega \times (0,T))$ stated by Lemma \ref{lem:transuT} in Appendix.

\medskip
We now check Assumption \eqref{hyp:x}.
For $n \in \llbracket 0, N\exm \rrbracket$, $P \in \mathcal P\exm_{\mathrm{int}}$ and $\zeta \in \edgespart(P)$, let 
\[
R_{P,\zeta}^n = \frac 1 {|P|} \int_{t_n}^{t_{n+1}} \int_P \Bigl| \Big( (\bfF\exm)_{\zeta}^n -\bff(q\exm,\bfv\exm) (\bfx,t) \Big) \cdot \bfn_{P,\zeta} \Bigr| \dx \dt
\]
and let 
\[
R\exm = \sum_{n=0}^{N\exm-1} \sum_{P \in \mathcal P\exm_{\mathrm{int}}}   {\diam(P)}  \sum_{\zeta \in \edgespart} |  \zeta |  \ R_{P,\zeta}^n.
\]

\medskip
We now express $R_{P,\zeta}^n$, for the RT and MAC discretisations successively.
\begin{list}{-}{\itemsep=0.5ex \topsep=0.5ex \leftmargin=1.cm \labelwidth=0.3cm \labelsep=0.5cm \itemindent=0.cm}
\item RT case -- In the case of general quadrangular meshes with the whole vector unknowns located on the edges, we have
\[
\bfF\exm)_{\zeta}^n = g(q_\zeta^n)\, \bfv_\zeta^n \quad \mbox{and} \quad
\bff(q\exm,\bfv\exm)(\bfx,t) = g(q_P^n)\ \bfv_{\zeta'}^n \mbox{ for } \bfx \in D_{P,\zeta'},\ \zeta' \in \edgespart(P).
\]
We thus get, denoting by $\vert \bfa \vert$ the Euclidean norm of any vector $\bfa \in \xR^2$,
\[
\Big| \Big((\bfF\exm)_{\zeta}^n - \bff(U\exm)(\bfx,t)\Big)\cdot \bfn_{P,\zeta} \Big| = \bigl| g(q_\zeta^n)\, \bfv_\zeta^n - g(q_P^n)\ \bfv_{\zeta'}^n \bigr|
\mbox{ for } \bfx \in D_{P,\zeta'},\ \zeta' \in \edgespart(P).
\]
Let $Q$ the primal cell such that $\zeta = P|Q$.
Since $q_\zeta^n$ is a convex combination of $q_P^n$ and $q_Q^n$, we thus get, for $\bfx \in P$,  and $t \in [t_n t_{n+1})$,
\[
\Big| \Big((\bfF\exm)_{\zeta}^n - \bff(U\exm)(\bfx,t)\Big)\cdot \bfn_{P,\zeta} \Big| \leq C\ \Bigl(|q^n_P - q^n_Q| + \sum_{\zeta' \in \edgespart(P)} \vert \bfv_\zeta^n-\bfv_{\zeta'}^n \vert\Bigr),
\]
where $C$ only depends on $\Vert q\exm \Vert_{L^\infty(\Omega \times (0,T))}$, $\Vert \bfv\exm \Vert_{L^\infty(\Omega \times (0,T))^2}$ and $C_g$.
Integrating over $P \times (t_n,t_{n+1})$, we obtain
\[
R_{P,\zeta}^n \leq C\ (t_{n+1}- t_n) \Bigl(|q^n_P - q^n_Q| + \sum_{\zeta' \in \edgespart(P)} \vert \bfv_\zeta^n-\bfv_{\zeta'}^n \vert\Bigr).
\]

\medskip
\item MAC case -- In this case, the velocity components are piecewise constant on different grids.
Let $i$ be the index such that $\zeta \in \edgespart^{(i)}$, and let $\zeta'$ be the other edge of $P$ normal to $\bfe^{(i)}$, \ie\ the opposite of $\zeta$ in $P$.
We have
\[
\bfF\exm)_{\zeta}^n \cdot \bfn_{P,\zeta} = g(q_\zeta^n)\, v_\zeta^n \ \delta_\zeta \quad \mbox{and} \quad
\bff(q\exm,\bfv\exm)(\bfx,t) =
\begin{cases}
g(q_P^n) \ v_{\zeta}^n\ \delta_\zeta \mbox{ if }\bfx \in D_{P,\zeta},
\\[1ex] 
g(q_P^n) \  v_{\zeta'}^n\ \delta_\zeta \mbox{ if } \bfx \in D_{P,\zeta'}, 
\end{cases} 
\]
with $\delta_\zeta = \bfn_{P,\zeta} \cdot \bfe^{(i)}$.
We thus get
\[
\Big| \Big((\bfF\exm)_{\zeta}^n - \bff(U\exm)(\bfx,t)\Big)\cdot \bfn_{P,\zeta} \Big| =
\begin{cases}
\bigl| g(q_\zeta^n)\, \bfv_\zeta^n - g(q_P^n) \ v_{\zeta}^n \bigr| \mbox{ if }\bfx \in D_{P,\zeta},
\\[1ex] 
\bigl| g(q_\zeta^n)\, \bfv_\zeta^n - g(q_P^n) \  v_{\zeta'}^n \bigr| \mbox{ if } \bfx \in D_{P,\zeta'}, 
\end{cases} 
\]
and hence, for $\bfx \in P$,  and $t \in [t_n t_{n+1})$, denoting by $Q$ the primal cell such that $\zeta = P|Q$,
\[
\Big| \Big((\bfF\exm)_{\zeta}^n - \bff(U\exm)(\bfx,t)\Big)\cdot \bfn_{P,\zeta} \Big| \leq C\ \bigl( |q_P^n - q_Q^n | +  \ |v_\zeta^n - v_{\zeta'}^n | \bigr),
\]
where $C$ only depends on $\Vert q\exm \Vert_{L^\infty(\Omega \times (0,T))}$, $\Vert \bfv\exm \Vert_{L^\infty(\Omega \times (0,T))^2}$ and $C_g$.
Therefore, integrating over $P \times (t_n,t_{n+1})$, we finally get
\[
R_{P,\zeta}^n \leq C\ (t_{n+1}- t_n)\ \Bigl(|q^n_P - q^n_Q| +   \vert \bfv_\zeta^n-\bfv_{\zeta'}^n \vert \Bigr).
\]
\end{list}
Note that, in these computations, we have not addressed the case where $\zeta$ is an external edge, taking benefit of the fact that, in the expression of $R\exm$, the sum is retricted to the internal cells.

\medskip
From the definition of $R\exm $, we thus get that, for both cases, it satisfies the following inequality:
\[
R\exm \leq C\ \bigl( R_1\exm + R_2\exm \bigr),
\]
with
\begin{equation*} 
R_1\exm =  \ \sum_{n=0}^{N\exm-1} (t_{n+1} - t_n) \sum_{P \in \mathcal P\exm} \diam(P) \ \sum_{\substack{\zeta \in \edgespart(P), \\ \zeta=P|Q}} |\zeta|\ |q^n_P - q^n_Q|,
\end{equation*} 
and 
\begin{equation*} R_2\exm = 
\begin{cases} 
\displaystyle \sum_{n=0}^{N\exm-1} (t_{n+1} - t_n)  \sum_{P \in \mathcal P\exm} \diam(P) \sum_{(\zeta, \zeta') \in \edgespart(P)^2} (|\zeta| + |\zeta'|)\ \vert \bfv_\zeta^n-\bfv_{\zeta'}^n \vert
&
\mbox{in the RT case},
\\[2ex] \displaystyle
\sum_{n=0}^{N\exm-1} (t_{n+1} - t_n)  \sum_{P \in \mathcal P\exm} \diam(P) \sum_{\substack{i=1,2, \\ (\zeta, \zeta') \in \edgespart^{(i)}(P)^2}}  (|\zeta| + |\zeta'|)\ \vert v_\zeta^n-v_{\zeta'}^n \vert 
& 
\mbox{in the MAC case}.
\end{cases}
\end{equation*} 
There only remains to prove that $R_1\exm$  and $R_2\exm$ tend to zero as  $m$ tends to $+\infty$. 
Reordering the summation in $R_1\exm$, we get that 
\[
R_1\exm = \sum_{n=0}^{N\exm-1} (t_{n+1} - t_n) \sum_{P \in \mathcal P\exm}
\   \sum_{\substack{\zeta \in \edgespart(P), \\ \zeta=P|Q}} \omega_\zeta\ |q^n_P - q^n_Q|, \quad
\mbox{with } \omega_\zeta = \Bigl(\diam(P) + \diam(Q)\Bigr)\ |\zeta|.
\]
Lemma \ref{lem:transuT} states that $R_1\exm$ tends to zero if the weight $\omega_\zeta$ is controlled by both $|P|$ and $|Q|$; since we have $\omega_\zeta \leq 2 \bigl(\max(\diam(P),\diam(Q))\bigr)^2$, this is easily obtained using Assumption \eqref{eq:hyp_reg}. 

\begin{figure}[tb]
\begin{center}
\scalebox{0.9}{
\begin{tikzpicture}
\fill[color=bclair,opacity=0.3] (1.,1.) --  (2.5,2.5) -- (0.5,4.2) -- (1.,1.);
\fill[color=orangec!30!white,opacity=0.3] (1.,1.) -- (4.,1.) -- (2.5,2.5) -- (1.,1.);
\draw[very thick, color=black] (1.,1.) -- (4.,1.) node[midway,sloped,below]{$\zeta$};
\draw[very thick, color=black] (0.5,1.1) -- (1.,1.) -- (0.9,0.5) ;
\draw[very thick, color=black] (0.,4.2) -- (0.5,4.2) -- (0.5,4.7) ;
\draw[very thick, color=black] (4.1,4.5) -- (4.,4.) -- (4.5,3.9);
\draw[very thick, color=black] (4.,0.5) -- (4.,1.)-- (4.5,1.);
\draw[very thick, color=black] (1.,1.) -- (0.5,4.2) node[midway,left]{$\zeta'$};
\draw[very thin, color=black] (1.,1.) -- (2.5,2.5) node[midway,sloped,above]{$\quad \eta = \zeta|\zeta'$};
\draw[very thick, color=black] (1.,1.) -- (0.5,4.2) -- (4.,4.) -- (4.,1.)-- (1.,1.);
\draw[thick, color=bfonce] (1.,1.) -- (2.5,2.5) -- (0.5,4.2);
\draw[thick, color=orangec!80!black] (1.,1.) -- (2.5,2.5)-- (4.,1.);
\draw[thick, color=rougef!80!white] (2.5,2.5) -- (4.,4.);
\end{tikzpicture}
\hspace{2cm}
\begin{tikzpicture}
\fill[color=bclair,opacity=0.3]   (2.5,2.5) -- (0.5,4.2) -- (4.,4.) -- (2.5,2.5);
\fill[color=orangec!30!white,opacity=0.3] (1.,1.) -- (4.,1.) -- (2.5,2.5) -- (1.,1.);
\draw[very thin, color=black] (1.,1.) -- (4.,1.) node[midway,sloped,below]{$\zeta$};
\draw[very thick, color=black] (1.,1.) -- (0.5,4.2) node[midway,left]{$\zeta''$};
\draw[very thin, color=black] (0.5,4.2)-- (4.,4.) node[midway,sloped,above]{$\zeta'$};
\draw[very thin, color=black] (1.,1.) -- (2.5,2.5) node[midway,sloped,above]{$\quad \eta = \zeta|\zeta''$};
\draw[very thin, color=black] (0.5,4.2)  -- (2.5,2.5) node[midway,sloped,above]{$\quad \eta' = \zeta'|\zeta''$};
\draw[very thick, color=black] (1.,1.) -- (0.5,4.2) -- (4.,4.) -- (4.,1.)-- (1.,1.);
\draw[very thick, color=black] (0.5,1.1) -- (1.,1.) -- (0.9,0.5) ;
\draw[very thick, color=black] (0.,4.2) -- (0.5,4.2) -- (0.5,4.7) ;
\draw[very thick, color=black] (4.1,4.5) -- (4.,4.) -- (4.5,3.9);
\draw[very thick, color=black] (4.,0.5) -- (4.,1.)-- (4.5,1.);
\draw[thick, color=bfonce] (4.,4.) -- (2.5,2.5) -- (0.5,4.2) ;
\draw[thick, color=orangec!80!black] (1.,1.) -- (2.5,2.5)-- (4.,1.) ;
\end{tikzpicture}
}
\end{center}
\caption{Left: the primal edges $\zeta$ and $\zeta'$ are adjacent. Right: the primal edges $\zeta$ and $\zeta'$ are opposite.}
\label{fig:zetazetap}
\end{figure}
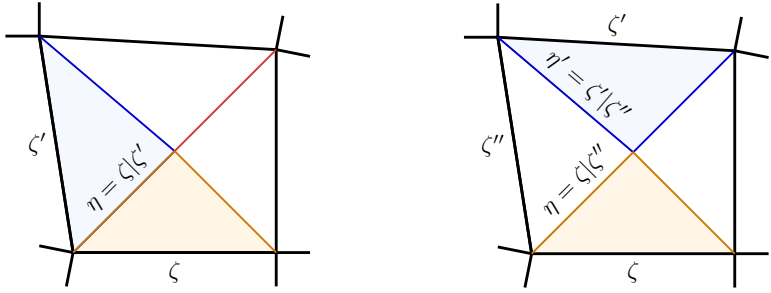

\medskip
As to the term $R_2\exm$, let us start by the RT case. 
We distinguish two cases for the pairs $(\zeta, \zeta') \in \edgespart(P)^2$ that appear in the summation: either the dual cells $D_\zeta$ and $D_\zeta'$ share a common (dual) edge $\eta=\zeta|\zeta'  \in \edgespart^\ast$, where $\edgespart^\ast$ denotes the set of edges of the dual mesh, or they are opposite edges in the quadrilateral cell $P$; in this latter case, we may write that 
\[
|\bfv_{\zeta}^n - \bfv_{\zeta'}^n| \le |\bfv_{\zeta}^n - \bfv_{\zeta''}^n|  + |\bfv_{\zeta''}^n - \bfv_{\zeta'}^n|, 
\]
where $\zeta'' \in \edgespart(P)$ is such that the dual cell $D_{\zeta''}$  shares a common (dual) edge $\eta$ (resp.  $\eta'$) $\in  \edgespart^\ast$  with $D_{\zeta}$  (resp. $D_{\zeta'}$) as shown in Figure \ref{fig:zetazetap}.
We have:
\[
\sum_{(\zeta, \zeta') \in \edgespart(P)^2} (|\zeta| + |\zeta'|)\ \vert \bfv_\zeta^n-\bfv_{\zeta'}^n \vert \leq 2\,\diam(P) \sum_{(\zeta, \zeta') \in \edgespart(P)^2}\vert \bfv_\zeta^n-\bfv_{\zeta'}^n \vert,
\]
and, since the decompositions of the jumps needed for pairs of opposite edges make the jump between two adjacent faces appears only a bounded number of times,
\[
\sum_{(\zeta, \zeta') \in \edgespart(P)^2} (|\zeta| + |\zeta'|)\ \vert \bfv_\zeta^n-\bfv_{\zeta'}^n \vert \leq C\ \diam(P) \sum_{\eta = \zeta|\zeta'\in \edgespart^\ast(P)} \vert \bfv_\zeta^n-\bfv_{\zeta'}^n \vert,
\]
with $C$ a given integer number and $\edgespart^\ast(P)$ the edges of the dual mesh included in $P$.
We thus get
\[
R_2\exm  \le C\ \sum_{n=0}^{N\exm-1} (t_{n+1} - t_n)  \sum_{\eta = \zeta|\zeta'\in \edgespart^\ast} \diam(P_\eta)^2\ \vert \bfv_\zeta^n-\bfv_{\zeta'}^n \vert,
\]
where $P_\eta$ stands for the primal cell in which $\eta$ is included.
The right-hand side of this inequality is thus a collection of jumps across the dual edges, with, for an edge $\eta$, a weight given by
\[
\omega_\eta = C\ \diam(P_\eta)^2.
\]
Thanks to Lemma \ref{lem:transuT}, $R_2\exm$ tends to zero when $m$ tends to $+\infty$ if $\omega_\eta$ is controlled by both $|D_\zeta|$ and $|D_{\zeta'}|$; this is indeed the case thanks to Assumption \eqref{eq:hyp_reg}, since $|D_\zeta| \geq |P_\eta|/4$ and $|D_{\zeta'}| \geq |P_\eta|/4$.

\medskip
Let us now turn to the MAC case, which is in fact simpler; indeed, the differences of velocities appearing in the expression of $R_2\exm$ are all jumps across dual edges, and we may thus recast $R_2\exm$ as
\[
R_2\exm  = \sum_{n=0}^{N\exm-1} (t_{n+1} - t_n) \sum_{i=1}^2 \sum_{\eta = \zeta|\zeta'\in (\edgespart^{(i)})^\ast}  \diam(P_\eta)\ (|\zeta| + |\zeta'|)\ \vert v_\zeta^n-v_{\zeta'}^n \vert,
\]
where $(\edgespart^{(i)})^\ast$ denotes the set of edges of the $i$-th dual mesh and $P_\eta$ is the primal cell in which lies $\eta$.
We thus again have a collection of jumps across the dual edges, with, for an edge $\eta$ included in a primal cell $P_\eta$ and separating the dual cells $D_\zeta$ and $D_{\zeta'}$, a weight given by
\[
\omega_\eta =  \diam(P_\eta) \ \bigl(|\zeta| + |\zeta'|\bigr).
\]
Thus, again thanks to Lemma \ref{lem:transuT}, $R_2\exm$ tends to zero when $m$ tends to $+\infty$ since, remarking that $|D_\zeta| \geq |P_\eta|/2$, $|D_{\zeta'}| \geq |P_\eta|/2$ and $\omega_\eta \leq 2 \diam(P_\eta)^2$, the weight $\omega_\eta$ is controlled by both $|D_\zeta|$ and $|D_{\zeta'}|$ thanks to Assumption \eqref{eq:hyp_reg}.
\end{proof}

\begin{remark}[On the required regularity of the time discretisation]
The assumption $\theta_3(\tdisc\exm) \leq \theta$, for $m \in \xN$, may be avoided thanks to a different choice of the interpolation of the test function (see Remark \ref{rem:interpm}).
However, this assumption is very mild (in fact, we do not have in mind any scheme where the ratio between two consecutive time-steps is likely to blow up when refining the discretisation).
\end{remark}
%
%
\appendix
\section{Convergence of discrete functions in $L^1$}

We recall a result proven in \cite[Lemma 4.3]{gal-19-wea}.
To facilitate its use in the proof of Lemma \ref{lem:cons_stag}, it is rephrased here under a slightly general form than in \cite{gal-19-wea} (see Remark \ref{rmrk:diff} below for the differences).

\medskip
Let $\mathcal M$ be a conforming mesh of the domain $\Omega$ of $\xR^d$, $d=1,2,3$, in polygonal or polyhedral subsets, and $\tdisc=(t_i)_{i \in \llbracket 0, N \rrbracket}$ be a time discretisation of the interval $(0,T)$, \ie\ a sequence of real numbers such that $0=t_0 < \dots < t_n < \dots t_N=T$.
We denote by $\delta t_\tdisc$ the time step, defined by $\delta t_\tdisc = \max \{t_{n+1}-t_n,\ n \in \llbracket 0, N-1 \rrbracket\}$.
For $u \in L^1(\Omega \times (0,T))$, $K \in \mesh$ and $n$ such that $n \in \llbracket 0, N-1 \rrbracket$, let $u_K^n$ be the mean value of $u$ over $K \times (t_n,t_{n+1})$.
We denote by $\edgesint$ the internal faces of the mesh and the face $\edge \in \edgesint$ separating the cells $K$ and $L$ is denoted by $\edge=K|L$. 
We define the following quantity:
\begin{equation}\label{transT-u}
T_{\mesh,\tdisc} \, u= \sum_{n=0}^{N-1} (t_{n+1}-t_n) \sum_{\substack{\edge \in \edgesint \\ \edge=K|L}} \omega_\edge\ |u_K^n-u_L^n|
+ \sum_{n=0}^{N-2} \delta_{n+1/2}\ \sum_{K\in\mesh} |K|\ |u_K^{n+1}-u_K^n|,
\end{equation}
where $(\omega_\edge)_{\edge \in \edgesint}$ and $(\delta_{n+1/2})_{n \in \llbracket 0, N-2 \rrbracket}$ are two sets of non-negative weights.
We introduce the two following parameters:
\begin{equation} \label{eq:reg}
\theta_\mesh=\max_{K \in \mesh} \max_{\edge \in \edgesint(K)}\ \frac {\omega_\edge}{|K|}, \quad
\theta_{\tdisc}= \max_{n \in \llbracket 0, N-2 \rrbracket}\ \Bigl\{ \frac{\delta_{n+1/2}}{t_{n+1} - t_n},\ \frac{\delta_{n+1/2}}{t_{n+2} - t_{n+1}}\Bigl\},
\end{equation}
with $\edgesint(K)$ the set of internal faces of $K$.
We denote by $\delta(\mesh)$ the space step characterizing $\mesh$, \ie\ $\delta(\mesh) = \max_{K \in \mesh} \diam(K)$.
Then the following convergence result holds.

\begin{lemma}\label{lem:transuT}
Let $\theta >0$  and $(\mesh\exm)_{m \in \xN}$ be a sequence of meshes  and for each $m \in \xN$, $\theta_{\mesh\exm}$ defined by \eqref{eq:reg}.
We assume that $\theta_{\mesh\exm} \le \theta$ for all $m \in \xN$ and $\lim_{m \to +\infty} \delta(\mesh\exm) =0$.
We suppose that the number of faces of a cell $K \in \mesh\exm$ is bounded by $\mathcal N_\edges$, for any $m \in \xN$.
For $m \in \xN$, we suppose given a time discretisation $\tdisc\exm$, and suppose that $\delta t_{\tdisc \exm}$ also tends to zero when $m$ tends to $+\infty$, and that $\theta_{\tdisc\exm} \leq \theta$ for all $m \in \xN$.
Let $u \in L^1(\Omega\times(0,T))$ and $(u_p)_{p \in \xN}$ be a sequence of functions of $L^1(\Omega\times(0,T))$ such that $u_p \to u$ in $L^1(\Omega\times(0,T))$ as $p \to +\infty$.\\[0.5ex]
Then $T_{\mesh\exm,\tdisc\exm}\, u_p$ defined by \eqref{transT-u} tends to zero when $m$ tends to $+\infty$ uniformly with respect to $p \in \xN$.
\end{lemma}

\begin{remark} \label{rmrk:diff}
The difference between Lemma \ref{lem:transuT} and the formulation of the same convergence result in \cite{gal-19-wea} lies in the definition of the weight of the jumps, which is more general in Lemma \ref{lem:transuT}.
In fact,  even though the volume of the dual cells are present in the quantity of interest in \cite[Lemma 4.3]{gal-19-wea}, the proof itself does not require the introduction of a dual mesh to define the weight of the jumps through the faces featured in the definition of $T_{\mesh,\tdisc}\, u$. 
Therefore, Lemma \cite[Lemma 4.3]{gal-19-wea} readily extends to the framework of Lemma \ref{lem:transuT}.
\end{remark}

\medskip
This generalization is in most cases sufficient.
However, we may go one step further, still with minor modifications of the proof of \cite{gal-19-wea}, as follows.
Let $\mathcal S_x$ be a set of  cardinal 2 - subsets of $\mathcal M$, and $\mathcal S_t$ be a set of  cardinal 2 - subsets of $\llbracket 0, N \rrbracket$.
Let $\widetilde T_{\mesh,\tdisc}\, u$ be defined by
\begin{equation}\label{transT-u_gene}
\widetilde T_{\mesh,\tdisc}\, u= \sum_{n=0}^{N-1} (t_{n+1}-t_n) \sum_{\{K,L\} \in \mathcal S_x} \omega_{K,L}\ |u^{n+1}_L - u^{n+1}_K|
+ \sum_{\{p,q\} \in \mathcal S_t}\ \delta_{p,q}\ \sum_{K\in\mesh} |K|\ |u^p_K - u^q_K|,
\end{equation}
where $(\omega_{K,L})_{\{K,L\} \in \mathcal S_x}$ and $(\delta_{p,q})_{\{p,q\} \in \mathcal S_t}$ are two sets of non-negative weights.
We introduce the two following parameters:
\begin{equation} \label{eq:reg_gene}
\begin{array}{l} \displaystyle
\theta_\mesh=\max_{K \in \mesh} \frac 1{|K|}\ \sum_{\substack{L \in \mathcal M\\ \{K,L\} \in \mathcal S_x}} \omega_{K,L},\qquad
\theta_{\tdisc}= \max_{n \in \llbracket 0, N-1 \rrbracket}\ \frac 1 {t_{n+1} - t_n}\ \sum_{\substack{p \in\llbracket 0, N \rrbracket\\{ \{n,p\} \in \mathcal S_t}}} \delta_{n,p}.
\end{array}
\end{equation}
For $\{K,L\} \in \mathcal S_x$ and $\{p,q\}\in \mathcal S_t$, let
\[
\mathfrak d (\{K,L\}) = \max_{(\bfx,\bfy) \in K \times L} |\bfy-\bfx|, \quad \mathfrak d(\{p,q\})=\begin{cases} t_{q+1}-t_p &\mbox{ if } q >p,\\
t_{p+1}-t_q &\mbox{ otherwise } \end{cases}
\]
and let
\[
\mathfrak d(\mesh) = \max_{\{K,L\} \in \mathcal S_x} \mathfrak d (\{K,L\}),\quad \mathfrak d(\tdisc) = \max_{\{p,q\} \in \mathcal S_t} \mathfrak d (\{p,q\}).
\]
Then the following convergence result holds.

\begin{lemma}\label{lem:transuT_gene}
Let $(\mesh\exm)_{m \in \xN}$ and $(\tdisc\exm)_{m \in \xN}$ be a given sequence of meshes and time discretisations.
Let us suppose there exists $\theta > 0$ such that $\theta_{\mesh\exm} \le \theta$ and $\theta_{\tdisc\exm} \le \theta$ for all $m \in \xN$, with $\theta_{\mesh\exm}$ and $\theta_{\tdisc\exm}$ given by Equation \eqref{eq:reg_gene}.
Let us assume that $\mathfrak d(\mesh\exm)$ and $\mathfrak d(\tdisc\exm)$ tend to zero when $m$ tends to $+\infty$.
Let $u \in L^1(\Omega\times(0,T))$ and $(u_p)_{p \in \xN}$ be a sequence of functions of $L^1(\Omega\times(0,T))$ such that $u_p \to u$ in $L^1(\Omega\times(0,T))$ as $p \to +\infty$.\\[0.5ex]
Then $\widetilde T_{\mesh\exm,\tdisc\exm}\, u_p$ defined by \eqref{transT-u_gene} tends to zero when $m$ tends to $+\infty$ uniformly with respect to $p \in \xN$.
\end{lemma}
%
%
\bibliographystyle{abbrv}
\bibliography{lw}
\end{document}